\documentclass[12pt]{article}
\usepackage{amsmath}
\usepackage{amssymb}

\newcommand{\eop}{\bigstar}  

\newcommand{\dom}{{\rm dom}}

\newcommand{\OO}{{\mathcal O}}

\newenvironment{proof}{\noindent{\bf Proof.}}{\par\bigskip}
\newenvironment{Proof}{\noindent{\bf Proof.}}{\par\bigskip}

\newtheorem{THEOREM}{Theorem}[section]

\newtheorem{Conclusion}[THEOREM]{Conclusion}

\newtheorem{Hypothesis}[THEOREM]{Hypothesis}

\newtheorem{LEMMA}[THEOREM]{Lemma}

\newtheorem{Main Theorem}[THEOREM]{Main Theorem}
\newenvironment{main Theorem}{\begin{Main Theorem}}
{\end{Main Theorem}}

\newtheorem{Theorem}[THEOREM]{Theorem}

\newtheorem{Definition}[THEOREM]{Definition}

\newtheorem{Conventions}[THEOREM]{Conventions}

\newtheorem{Main Definition}[THEOREM]{Main Definition}
\newenvironment{main definition}{\begin{Main Definition}}
{\end{Main Definition}}

\newtheorem{Lemma}[THEOREM]{Lemma}

\newtheorem{Notation}[THEOREM]{Notation}

\newtheorem{Convention}[THEOREM]{Convention}

\newtheorem{Note}[THEOREM]{Note}

\newtheorem{Observation}[THEOREM]{Observation}
\newenvironment{observation}{\begin{Observation}}
{\end{Observation}}

\newtheorem{Remark}[THEOREM]{Remark}

\newtheorem{Question}[THEOREM]{Question}

\newtheorem{Main Fact}[THEOREM]{Main Fact}
\newenvironment{main Fact}{\begin{Main Fact}}{\end{Main Fact}}

\newtheorem{Fact}[THEOREM]{Fact}

\newtheorem{Subfact}[THEOREM]{Subfact}

\newtheorem{Claim}[THEOREM]{Claim}
\newenvironment{claim}{\begin{Claim}}{\end{Claim}}

\newtheorem{Main Claim}[THEOREM]{Main Claim}
\newenvironment{main claim}{\begin{Main Claim}}{\end{Main Claim}}

\newtheorem{Crucial Claim}[THEOREM]{Crucial Claim}
\newenvironment{crucial claim}{\begin{Crucial Claim}}{\end{Crucial Claim}}

\newtheorem{Subclaim}[THEOREM]{Subclaim}

\newtheorem{Corollary}[THEOREM]{Corollary}

\newtheorem{Example}[THEOREM]{Example}

\newtheorem{Problem}[THEOREM]{Problem}

\newtheorem{Proposition}[THEOREM]{Proposition}

\newtheorem{Discussion}[THEOREM]{Discussion}

\newenvironment{Proof of the Subfact}
{\noindent{\bf Proof of the Subfact.}}{\par\bigskip}

\newenvironment{Proof of the Theorem}
{\noindent{\bf Proof of the Theorem.}}{\par\bigskip}

\newenvironment{Proof of the Proposition}
{\noindent{\bf Proof of the Proposition.}}{\par\bigskip}

\newenvironment{Proof of the Conclusion}
{\noindent{\bf Proof of the Conclusion.}}{\par\bigskip}

\newenvironment{Proof of the Observation}
{\noindent{\bf Proof of the Observation.}}{\par\bigskip}

\newenvironment{Proof of the Fact}
{\noindent{\bf Proof of the Fact.}}{\par\bigskip}

\newenvironment{Proof of the Lemma}
{\noindent{\bf Proof of the Lemma.}}{\par\bigskip}

\newenvironment{Proof of the Claim}
{\noindent{\bf Proof of the Claim.}}{\par\bigskip}

\newenvironment{Proof of the Corollary}
{\noindent{\bf Proof of the Corollary.}}{\par\bigskip}

\newenvironment{Proof of the Subclaim}
{\noindent{\bf Proof of the Subclaim.}}{\par\medskip}

\newenvironment{Proof of the Main Claim}
{\noindent{\bf Proof of the Main Claim.}}{\par\bigskip}

\newenvironment{Proof of the Crucial Claim}
{\noindent{\bf Proof of the Crucial Claim.}}{\par\bigskip}

\catcode`\@=11
\def\@begintheorem#1#2{\rm \trivlist \item[\hskip \labelsep{\bf #1\ #2.}]}
\def\@opargbegintheorem#1#2#3{\rm \trivlist
      \item[\hskip \labelsep{\bf #1\ #2\ (#3).}]}
\catcode`\@=12

\catcode`\@=11
\def\@&{\hskip2pt \&\hskip2pt}
\catcode`\@=12

\newcommand{\into}{\rightarrow}

\newcommand{\rest}{\upharpoonright}  

\newcommand{\KK}{{\cal K}}

\newcommand{\PP}{{\cal P}}

\newcommand{\concat}{\kern-.25pt\raise4pt\hbox{$\frown$}\kern-.25pt}

\newcommand{\piba}{\pi_\alpha^\beta}  
\newcommand{\pioa}{\pi^{\omega_1}_\alpha}

\newcount\skewfactor
\def\mathunderaccent#1#2 {\let\theaccent#1\skewfactor#2
\mathpalette\putaccentunder}
\def\putaccentunder#1#2{\oalign{$#1#2$\crcr\hidewidth
\vbox to.2ex{\hbox{$#1\skew\skewfactor\theaccent{}$}\vss}\hidewidth}}

\begin{document}

\title{$CH$, a problem of Rolewicz and bidiscrete systems}

\author{Mirna D\v zamonja\footnote{School of Mathematics, University of East Anglia, Norwich NR4 7TJ,
UK, h020@uea.ac.uk} \,\,and Istv\'an Juh\'asz\footnote{Alfr\'ed R\'enyi Institute of Mathematics,
Hungarian Academy of Sciences, H-1053 Budapest, Re\'altanoda u. 13-15., H-1364 Budapest, P. O. Box: 127, Hungary, juhasz@renyi.hu}}


\maketitle

\begin{abstract} We give a construction under $CH$ of a non-metrizable compact Hausdorff space $K$
such that any uncountable semi-biorthogonal sequence in $C(K)$ must be of a very specific kind.
The space $K$ has many nice properties, such as being hereditarily separable, hereditarily Lindel\"of and a 2-to-1
continuous preimage of a metric space, and all Radon measures on $K$ are separable.
However $K$ is not a Rosenthal compactum.

We introduce the notion of a {\em bidiscrete system} in a compact
space $K$. These are subsets of $K^2$ which determine biorthogonal
systems of a special kind in $C(K)$ that we call {\em nice}. We
note that every infinite compact Hausdorff space $K$ has a
bidiscrete system and hence a nice biorthogonal system of size
$d(K)$, the density of $K$. \footnote{M. D\v zamonja thanks EPSRC
for their support through the grant EP/G068720 and I. Juh\'asz
thanks the Hungarian National Foundation for Scientific Research
(OTKA) for their support from grants K 61600 and K 68262, as well
as to the University of East Anglia for supporting a one week
visit in November 2008. We both thank the Mittag-Leffler Institute
for their hospitality in September 2009 when part of this research
was done.}
\end{abstract}

\section{Introduction}\label{intro}
All topological spaces mentioned here are Hausdorff. As is traditional
in general Banach space theory, Banach spaces we mention
are considered as real Banach spaces even though nothing essential changes in the
context of complex spaces. Throughout let $K$ stand for an infinite compact topological space.

Let $X$ be a Banach space and $C$ a closed
convex subset of $X$.
A point $x_0\in C$ is a {\em point of support for} $C$ if there
is a functional $\varphi\in X^\ast$ such that $\varphi(x_0)\le\varphi(x)$
for all $x\in C$, and $\varphi(x_0)<\varphi(x')$ for some $x'\in C$.

Rolewicz  \cite{Ro} proved in 1978 that every separable closed convex
subset $Y$ of a Banach space contains a point
which is not a point of support for $Y$, and asked if every
non-separable Banach space must contain a closed convex set
containing only points of support. In fact, this topic was already considered
by Klee \cite{Klee} in 1955 and the above theorem follows from 2.6
in that paper, by the same proof and taking $x_i$s to form
a dense set in $C$ \footnote{We thank Libor Vesely for pointing
out this, not widely recognised, connection.}. However it was Rolewicz's paper
which started a whole series of articles on this topic, and his question
has not yet been settled completely. It is known that the answer to
Rolewicz's question is independent of ZFC, and it is still not known if
the negative answer follows from $CH$. In \S\ref{semi} we
construct a $CH$ example of a nonseparable Banach space of the form $C(K)$ which violates
a strenghtening of the requirements in the Rolewicz's question.

The proof in \S\ref{semi} uses certain systems of pairs of points
of $K$, whose structure seems to us to be of independent interest.
They appear implicitly in many proofs about biorthogonal systems
in spaces of the form $C(K)$, see \cite{Hajekbook}, but their
existence is in fact entirely a property of the compact space $K$.
We call such systems {\em bidiscrete systems}. They are studied in
\S \ref{bidiscr}. Specifically, we prove in Theorem \ref{lemma2}
that if $K$ is an infinite compact Hausdorff space then $K$ has a
bidiscrete system of size $d(K)$, the density of $K$. This theorem
has not been stated in this form before, but we note that an
argument by Todor\v cevi\'c in \cite{StevoMM} can be easily
extended to give this result.

We now give some historical background. Mathematical background will be presented in
Section \ref{background}.

Borwein and Vanderwerff \cite{Bv} proved in 1996 that, in a Banach
space $X$, the existence of a closed convex set all of whose
points are support points is equivalent to the existence of an
uncountable semi-biorthogonal sequence for $X$, where
semi-biorthogonal sequences are defined as follows:

\begin{Definition}\label{biorthognal} Let $X$ be a Banach space. A sequence
$\langle (f_\alpha, \varphi_\alpha):\,\alpha<\alpha^\ast\rangle$ in $X\times X^\ast$ is said
to be
a {\em semi-biorthogonal sequence} if for all $\alpha,\beta<\alpha^\ast$ we have:
\begin{itemize}
\item $\varphi_\alpha(f_\alpha)=1$,
\item $\varphi_\alpha (f_\beta)=0$ if $\beta<\alpha$,
\item $\varphi_\alpha (f_\beta)\ge 0$ if $\beta>\alpha$.
\end{itemize}
\end{Definition}

We remind the reader of the better known notion of {\em a biorthogonal system}
$\{(f_\alpha, \varphi_\alpha):\,\alpha<\alpha^\ast\}$ in $X\times X^\ast$ which is defined to satisfy the first
item of the Definition \ref{biorthognal} with the last two items are strengthened to
\begin{itemize}
\item $\varphi_\alpha (f_\beta)=0$ if $\beta\neq\alpha$.
\end{itemize}
Notice that the requirements of a semi-biorthogonal sequence make it clear that we really need a well ordering
of the sequence
in the definition, but that the definition of a biorthogonal system does not require an underlying  well-ordering.
There is nothing special about the values 0 and 1 in the above definitions, of course, and we could
replace them by any pair $(a,b)$ of distinct values in $\mathbb R$ and even let $b=b_\alpha$ vary with
$\alpha$. Equally, we could require the range of all $f_\alpha$ to be in $[0,1]$ or some other fixed nonempty
closed interval.

Obviously, any well-ordering of a biorthogonal system gives a
semi-biorthogonal sequence. On the other hand, there is an example
by Kunen under $CH$ of a nonmetrizable compact scattered space $K$
for which $X=C(K)$ does not have an uncountable biorthogonal
system, as proved by Borwein and Vanderwerff in \cite{Bv}. Since
$K$ is scattered, it is known that $X$ must have an uncountable
semi-biorthogonal sequence  (see \cite{Hajekbook} for a
presentation of a similar example under $\clubsuit$ and a further
discussion). Let us say that a Banach space is a {\em Rolewicz
space} if it is nonseparable but does not have an uncountable
semi-biorthogonal sequence.

In his 2006 paper \cite{StevoMM}, Todor\v cevi\'c proved that
under Martin's Maximum (MM) every non-separable Banach space has
an uncountable biorthogonal system, so certainly it has an
uncountable semi-biorthogonal sequence. Hence, under MM there are
no Rolewicz spaces. On the other hand, Todor\v cevi\'c informed us
that he realized in 2004 that a forcing construction in \cite{bgt}
does give a consistent example of a Rolewicz space. Independently,
also in 2004 (published in 2009), Koszmider gave a similar forcing construction in
\cite{Kosz}. It is still not known if there has to be a Rolewicz
space under CH.

Our motivation was to construct a Rolewicz space of the form
$X=C(K)$ under CH. Unfortunately, we are not able to do so, but we
obtain in Theorem \ref{CH} a space for which we can at least show
that it satisfies most of the known necessary conditions for a
Rolewicz space and that it has no uncountable semi-bidiscrete
sequences of the kind that are present in the known failed
candidates for such a space, for example in $C(S)$ where $S$ is
the split interval.

Specifically, it is known that if $K$ has a non-separable Radon
measure or if it is scattered then $C(K)$ cannot be Rolewicz
(\cite{GJM}, \cite{Hajekbook}) and our space does not have either
of these properties. Further, it is known that a compact space $K$
for which $C(K)$ is a Rolewicz space must be both HS and HL
(\cite{Lazar}, \cite{Bv}) while not being metrizable, and our
space has these properties, as well. It follows from the
celebrated structural results on Rosenthal compacta by Todor{\v
c}evi\'c in \cite{StevoRos} that a Rosenthal compactum cannot be a
Rolewicz space, and our space is not Rosenthal compact. Finally,
our space is not metric but it is a 2-to-1 continuous preimage of
a metric space. This is a property possessed by the forcing
example in \cite{Kosz} and it is interesting because of a theorem
from \cite{StevoRos} which states that every non-metric Rosenthal
compact space which does not contain an uncountable discrete
subspace is a 2-to-1 continuous preimage of a metric space. Hence
the example in \cite{Kosz} is a space which is not Rosenthal
compact and yet it satisfies these properties, and so is our
space.

\section{Background}\label{background}

\begin{Definition}\label{nice} Let $X=C(K)$ be the Banach space of continuous functions
on a compact space $K$. We say that a sequence $\langle
(f_\alpha,\phi_\alpha):\,\alpha<\alpha^\ast\rangle$ in $X\times
X^\ast$ is a {\em nice} semi-biorthogonal sequence if it is a
semi-biorthogonal sequence and there are points $\langle
x^l_\alpha:\,l=0,1,\alpha<\alpha^\ast\rangle$ in $K$ such that
$\phi_\alpha=\delta_{x^1_\alpha}-\delta_{x^0_\alpha}$, where
$\delta$ denotes the Dirac measure. We similarly define nice
biorthogonal systems.
\end{Definition}

As Definition \ref{nice} mentions points of $K$ and $C(K)$ does
not uniquely determine $K$~\footnote{see e.g. Miljutin's theorem
\cite{Mil1}, \cite{Mil2} which states that for $K$, $L$ both
uncountable compact metrizable, the spaces $C(K)$ and $C(L)$ are
isomorphic.}, the definition is actually topological rather than
analytic. We shall observe below that the existence of a nice
semi-biorthogonal sequence of a given length or of a nice
biorthogonal system of a given size in $C(K)$ is equivalent to the
existence of objects which can be defined in terms which do not
involve the dual $C(K)^\ast$.

\begin{Definition}\label{nicetop} (1) A system
$\{(x_\alpha^0, x_\alpha^1):\,\alpha<\kappa\}$ of pairs of points
in $K$ (i.e. a subfamily of $K^2$) is called {\em a bidiscrete
system in} $K$ if there exist functions $\{
f_\alpha:\,\alpha<\kappa\} \subseteq C(K)$ satisfying that for every
$\alpha,\beta<\kappa$:
\begin{itemize}
\item $f_\alpha(x_\alpha^l)=l$ for $l\in \{0,1\}$,
\item if $\alpha\neq\beta$ then $f_\alpha(x_\beta^0)=f_\alpha(x_\beta^1)$.
\end{itemize}
\end{Definition}

(2) We similarly define semi-bidiscrete sequences in $K$ as
sequences $\langle (x_\alpha^0,
x_\alpha^1):\,\alpha<\alpha^\ast\rangle$ of points in $K^2$ that
satisfy the first requirement of (1) but instead of the second the
following two requirements:
\begin{itemize}
\item if $\alpha>\beta$ then $f_\alpha(x_\beta^0)=f_\alpha(x_\beta^1)$,
\item if $\alpha<\beta$ then $f_\alpha(x_\beta^0)=1\implies f_\alpha(x_\beta^1)=1$.
\end{itemize}

\begin{Observation}\label{nice=bidiscrete} For a compact space $K$,
$\{(x_\alpha^0, x_\alpha^1):\,\alpha<\alpha^*\} \subseteq K^2$ is a
bidiscrete system  iff there are $\{ f_\alpha:\,\alpha<\alpha^*\}
\subseteq C(K)$ such that $\{(f_\alpha,
\delta_{x^\alpha_1}-\delta_{x^\alpha_0}):\,\alpha<\alpha^*\}$ is a
nice biorthogonal system for the Banach space $X=C(K)$. The
analogous statement holds for nice semi-bidiscrete sequences.
\end{Observation}

\begin{Proof} We only prove the statement for nice biorthogonal systems, the proof for the nice
semi-biorthogonal sequences is the same. If we are given a system
exemplifying (1), then
$\delta_{x^1_\alpha}(f_\beta)-\delta_{x^0_\alpha}(f_\beta)=
f_\beta(x^1_\alpha)-f_\beta(x^0_\alpha)$ has the values as
required. On the other hand, if we are given a nice biorthogonal
system of pairs $\{(f_\alpha,
\delta_{x^1_\alpha}-\delta_{x^0_\alpha}) : \alpha < \alpha^*\}$
for $X$, define for $\alpha<\alpha^*$ the function $g_\alpha \in
C(K)$ by $g_\alpha(x)=f_\alpha(x)-f_\alpha(x^0_\alpha)$. Then $\{
(x_\alpha^0, x_\alpha^1):\,\alpha<\alpha^\ast\}$ satisfies (1), as
witnessed by $\{ g_\alpha:\,\alpha<\alpha^\ast\}$.
$\eop_{\ref{nice=bidiscrete}}$
\end{Proof}

In the case of a 0-dimensional space $K$ we are often able to make
a further simplification by requiring that the functions
$f_\alpha$ exemplifying the bidiscreteness of $(x_\alpha^0,
x_\alpha^1)$ take only the values 0 and 1. This is clearly
equivalent to asking for the existence of a family $\{
H_\alpha:\,\alpha<\alpha^\ast\}$ of clopen sets in $K$ such that
each $H_\alpha$ separates $x^0_\alpha$ and $x^1_\alpha$ but not
$x^0_\beta$ and $x^1_\beta$ for $\beta\neq\alpha$. We call such
bidiscrete systems  {\em very nice}. We can analogously define a
{\em very nice} semi-bidiscrete sequence, where the requirements
on the clopen sets become $x^l_\alpha\in H_\alpha\iff l=1$,
$\beta<\alpha\implies [x^0_\beta\in H_\alpha\iff x^{1}_\beta\in
H_\alpha]$ and $[\beta>\alpha\, \wedge \, x^0_\beta\in H_\alpha]
\implies x^1_\beta\in H_\alpha$.

We shall use the expression {\em very nice (semi-)biorthogonal
system (sequence)} in $C(K)$ to refer to a nice
(semi-)biorthogonal system (sequence) obtained as in the proof of
Claim \ref{nice=bidiscrete} from a very nice (semi-)bidiscrete
system (sequence) in $K$.

\begin{Example}
\label{splitinterval} (1) Let $K$ be the split interval (or double
arrow) space, namely the ordered space $K=[0,1]\times\{0,1\}$,
ordered lexicographically.  Then $$\{\big((x,0), (x,1)\big) : x
\in [0,1]\}$$ forms a very nice bidiscrete system in $K$. This is
exemplified by the two-valued continuous functions $\{f_x : x \in
[0,1]\}$ defined by $f_x(r)=0$ if $r\le (x, 0)$ and $f_x(r)=1$
otherwise.

(2) Suppose that $\kappa$ is an infinite cardinal and
$K=2^\kappa$. For $l\in \{0,1\}$ and $\alpha<\kappa$ we define
$x^l_\alpha\in K$ by letting $x^l_\alpha(\beta)=1$ if
$\beta<\alpha$, $\,x^l_\alpha(\beta)=0$ if $\beta> \alpha$, and
$x^l_\alpha(\alpha)=l$. The clopen sets $H_\alpha=\{f\in
K:\,f(\alpha)=1\}$ show that the pairs $\{(x^0_\alpha,x^1_\alpha)
: {\alpha<\kappa}\}$ form a very nice bidiscrete system in the
Cantor cube $K = 2^\kappa$.
\end{Example}

In \cite{StevoMM}, Theorem 10, it is proved under
$MA_{\omega_1}\,$ that every Banach space of the kind $X=C(K)$ for
a nonmetrizable compact $K$ admits an uncountable nice
biorthogonal system. Moreover, at the end of the proof it is
stated that for a 0-dimensional $K$ this biorthogonal system can
even be assumed to be very nice (in our terminology).

As nice semi-biorthogonal sequences may be defined using only $K$
and $X=C(K)$ and do not involve the dual $X^\ast$, in
constructions where an enumerative tool such as $CH$ is used it is
easier to control nice systems than the general ones. In our CH
construction below of a closed subspace $K$ of $2^{\omega_1}$ we
would at least like to destroy all uncountable nice
semi-biorthogonal sequences by controlling semi-bidiscrete
sequences in $K$. We are only able to do this for semi-bidiscrete
sequences which are not already determined by the first
$\omega$-coordinates, in the sense of the following Definition
\ref{supernice} : In our space $K$ any uncountable nice
semi-biorthogonal sequence must be $\omega$-determined.

\begin{Definition}\label{supernice}
A family $\{(x^0_\alpha, x^1_\alpha):\,\alpha<\alpha^*\} \subseteq
2^{\omega_1} \times 2^{\omega_1}$ is said to be {\em
$\omega$-determined} if
\[
(\forall s\in
2^\omega)\,\{\alpha:\,x^0_\alpha\rest\omega=x^1_\alpha\rest\omega=s\}\mbox{
is countable}.
\]
For $K \subseteq 2^{\omega_1}$ we define an {\em $\omega$-determined
semi-biorthogonal sequence in $C(K)$} to be any nice
semi-biorthogonal sequence $\langle
(f_\alpha,\delta_{x^1_\alpha}-\delta_{x^0_\alpha}):\,\alpha<\alpha^\ast\rangle$
for which the associated semi-bidiscrete sequence $\langle
(x^0_\alpha, x^1_\alpha):\alpha<\alpha^\ast \rangle$ forms an
$\omega$-determined family.
\end{Definition}

 \section{The $CH$ construction}\label{semi}

\begin{Theorem}\label{CH} Under $CH$, there is a compact space $K \subseteq 2^{\omega_1}$ with the following properties:
\begin{itemize}
\item $K$ is not metrizable, but is a 2-to-1 continuous preimage of a metric space,
\item $K$ is HS and HL,
\item\label{treci} every Radon measure on $K$ is separable,
\item $K$ has no isolated points,
\item $K$ is not Rosenthal compact,
\item any uncountable nice semi-biorthogonal sequence in $C(K)$ is $\omega$-determined.
\end{itemize}
\end{Theorem}

\begin{proof} We divide the proof into two parts. In the first we give various requirements on the construction,
and show that if these requirements are satisfied the space meeting the claim of the theorem can be constructed. In the second part we show that
these requirements can be met.

\subsubsection{The requirements}
Our space will be a closed subspace of $2^{\omega_1}$. Every such space  can be viewed as the limit of an inverse system of spaces, as we now explain.

\begin{Definition}
For $\alpha\leq\beta\le\omega_1$, define $\piba : 2^{\beta}\rightarrow
2^{\alpha}$ by $\piba (f) = f\rest\alpha $.
\end{Definition}

Suppose that $K$ is a closed subspace of $2^{\omega_1}$, then
for $\alpha \le \omega_1$ we let $K_\alpha=\pi^{\omega_1}_\alpha(K)$.  So, if $\alpha \le \beta$ then
$K_\alpha$ is the $\pi^\beta_\alpha$-projection  of $K_\beta$. For $\alpha<\omega_1$ let
\[
A_\alpha=\pi^{\alpha+1}_\alpha(\{x\in K_{\alpha+1}:x(\alpha)=0\}), B_\alpha=\pi^{\alpha+1}_\alpha(\{x\in K_{\alpha+1}:x(\alpha)=1\}).
\]
The following statements are then true:

\begin{description}

\item{\bf R1}.\label{1.1} $K_\alpha$ is a closed subset of $2^\alpha$, and
$\pi_\alpha^\beta(K_\beta) = K_\alpha$ whenever
$\alpha \le \beta \le \omega_1$.

\item{{\bf R2}.}\label{1.2} For $\alpha < \omega_1$,
$A_\alpha$ and $B_\alpha$ are closed in $K_\alpha$,
$A_\alpha \cup B_\alpha = K_\alpha$, and
$K_{\alpha + 1} = A_\alpha \times \{0\} \cup B_\alpha \times \{1\}$.

\end{description}

Now $K$ can be viewed as the limit of the inverse system $\mathcal
K=\{K_\alpha:\,\alpha<\omega_1, \piba\rest K_\beta:\,\alpha\le
\beta<\omega_1\}$. Therefore to construct the space $K$ it is
sufficient to specify the system $\KK$, and as long as the
requirements {\bf R1} and {\bf R2} are satisfied, the resulting
space $K$ will be a compact subspace of $2^{\omega_1}$. This will
be our approach to constructing $K$, that is we define $K_\alpha$
by induction on $\alpha$ to satisfy various requirements that we
list as {\bf Rx}.

The property HS+HL will be guaranteed by a use of irreducible
maps, as in \cite{DzK}. Recall that for spaces $X,Y$, a map
$f:\,X\into Y$ is called {\em irreducible} on $A\subseteq X$ iff
for any proper closed subspace $F$ of $A$ we have that $f(F)$ is a
proper subset of $f(A)$. We shall have a special requirement to
let us deal with HS+HL, but we can already quote Lemma 4.2 from
\cite{DzK}, which will be used in the proof. It applies to any
space $K$ of the above form.

\begin{Lemma}\label{Lemma4.2} Assume that $K$ and $K_\alpha$ satisfy {\bf R1} and {\bf R2} above.
Then $K$ is HL+HS iff for all closed $H \subseteq K$,
there is an $\alpha <\omega_1$ for which
$\pioa$ is irreducible on $(\pioa )^{-1} (\pioa (H))$.
\end{Lemma}

In addition to the requirements given above
we add the following basic requirement {\bf R3} which assures that $K$ has no isolated points.

\begin{description}

\item{{\bf R3}.}\label{1.3} For $n < \omega$,
$K_n = A_n = B_n = 2^n$.  For $\alpha \ge \omega$, $A_\alpha$ and
$B_\alpha$ have no isolated points.

\end{description}

Note that the requirement {\bf R3} implies that for each $\alpha
\ge \omega$, $K_\alpha$ has no isolated points; so it is easy to
see that the requirements guarantee that $K$ is a compact subspace
of $2^{\omega_1}$ and that it has no isolated points. Further,
$K_\omega = 2^\omega$ by {\bf R1} and {\bf R3}. The space $K$ is
called {\em simplistic} if for all $\alpha$ large enough
$A_\alpha\cap B_\alpha$ is a singleton. For us `large enough' will
mean `infinite', i.e. during the construction we shall obey the
following:

\begin{description}
\item{{\bf R4}.}\label{1.4} For all $\alpha\in [\omega,\omega_1)$ we
have $A_\alpha\cap B_\alpha=\{s_\alpha\}$ for some $s_\alpha\in K_\alpha$.
\end{description}

By {\bf R4} we can make the following observation which will be useful later:

\begin{Observation}\label{delta} Suppose that $x\in K_\alpha, y\in K_\beta$ for some $\omega\le\alpha\le\beta$ and $x\nsubseteq y$,  $y\nsubseteq x$ with $\Delta(x,y)\ge\omega$. Then $x\rest\Delta(x,y)=y\rest\Delta(x,y)=s_{\Delta(x,y)}$.
\end{Observation}

As usual, we used here the notation $\Delta(x,y) = \min \{\alpha :
x(\alpha) \ne y(\alpha)\}$.

Requirement {\bf R4} implies that $K$ is not 2nd countable, hence not
metrizable.
The following is folklore in the subject, but one can also see \cite{Piotr} for a detailed explanation
and stronger theorems:

\begin{Fact} Every Radon measure on a simplistic space is separable.
\end{Fact}

Now we come back to the property HS+HL. To assure this we shall
construct an auxiliary Radon measure $\mu$ on $K$. This measure
will be used, similarly as in the proof from Section \S 4 in
\cite{DzK}, to assure that for every closed subset $H$ of $K$ we
have $H = (\pioa )^{-1} (\pioa (H))$ for some countable coordinate
$\alpha$. In fact, what we need for our construction is not the
measure $\mu$ itself but a sequence $\langle
\mu_\alpha:\,\alpha<\omega_1 \rangle$ where each $\mu_\alpha$ is a
Borel measure on $K_\alpha$ and these measures satisfy that for
each $\alpha\le\beta<\omega_1$ and Borel set $B\subseteq K_\beta$,
we have $\mu_\beta(B)=\mu_\alpha(\piba(B))$. As a side remark
the
sequence $\langle \mu_\alpha:\,\alpha<\omega_1 \rangle$ will
uniquely determine a Radon measure $\mu = \mu_{\omega_1}$ on $K$.
To uniquely determine each Borel (=Baire)
measure $\mu_\alpha$ it is sufficient to decide its values on the
clopen subsets of $K_\alpha$. We formulate a requirement to
encapsulate this discussion:

\begin{description}

\item{{\bf R5}.} For
$\alpha\le\omega_1$, $\mu_\alpha$ is a finitely additive probability
measure on the clopen subsets of $K_\alpha$, and
$\mu_\alpha = \mu_\beta (\piba )^{-1}$ whenever
$\omega\le\alpha \le \beta \le \omega_1$. For $\alpha \le \omega$,
$\mu_\alpha$ is the usual product measure on the clopen subsets of $K_\alpha=2^\alpha$.
\end{description}

Let $\widehat{\mu}_\alpha$ be the Borel measure on $K_\alpha$
generated by $\mu_\alpha$. It is easy to verify that {\bf R1}-{\bf
R5} imply that for $\alpha \le \omega$, $\widehat {\mu}_\alpha$ is
the usual product measure on $K_\alpha=2^\alpha$, and that for any
$\alpha$, $\widehat {\mu}_\alpha$ gives each non-empty clopen set
positive measure and measure 0 to each point in $K_\alpha$. We
shall abuse notation and use $\mu_\alpha$ for both
$\widehat{\mu}_\alpha$ and its restriction to the clopen sets.
Note that by the usual Cantor tree argument these properties
assure that in every set of positive measure there is an
uncountable set of measure 0; this observation will be useful
later on.

The following requirements will help us both to obtain HS+HL and
to assure that $K$ is not Rosenthal compact. To formulate these
requirements we use $CH$ to enumerate the set of pairs $\{(\gamma,
J):\,\gamma<\omega_1\,,\, J \subseteq 2^\gamma \mbox{ is Borel}
\}$ as $\{(\delta_\alpha, J_\alpha):\,\omega\le \alpha
<\omega_1\}$ so that $\delta_\alpha \le\alpha$ for all $\alpha$
and each pair appears unboundedly often.

Suppose that $\omega \le \alpha < \omega_1$ and $K_\alpha$ and $\mu_\alpha$ are defined.
We define the following subsets of $K_\alpha$:

{\parindent= 40pt

$C_\alpha = (\pi^\alpha_{\delta_\alpha} )^{-1} (J_\alpha )$, if $J_\alpha\subseteq
K_{\delta_\alpha}$; $C_\alpha = \emptyset$ otherwise.

$L_\alpha = C_\alpha$ if $C_\alpha$ is closed;
$L_\alpha = K_\alpha$ otherwise.

$Q_\alpha = L_\alpha \setminus\bigcup\{ O:
O \hbox{\ is open and \ }  \mu_\alpha (L_\alpha \cap O) = 0\}$

$N_\alpha =(L_\alpha\setminus Q_\alpha)\cup C_\alpha$, if
$\mu_\alpha(C_\alpha ) =0$;
$N_\alpha =(L_\alpha\setminus Q_\alpha)$ otherwise.

\medskip
}

Let us note that $L_\alpha$ is a closed subset of $K_\alpha$ and
that $Q_\alpha\subseteq L_\alpha$ is also closed and satisfies
$\mu_\alpha(Q_\alpha)=\mu_\alpha(L_\alpha)$, and hence
$\mu_\alpha(N_\alpha)=0$. Also observe that $Q_\alpha$ has no
isolated points, as points have $\mu_\alpha$ measure 0.

We now recall from \cite{DzK} what is meant by $A$ and $B$ being
{\em complementary regular closed subsets} of a space $X$: this
means that $A$ and $B$ are both regular closed with $A\cup B=X$,
while $A\cap B$ is nowhere dense in $X$. Finally, we state the
following requirements:

\begin{description}
\item{\bf R6}.\label{4.1}
For any $\beta\geq\alpha\geq\omega$, $s_\beta\notin
(\piba)^{-1}( N_\alpha)$;
\item{\bf R7}.\label{2.3} For any $\beta\geq\alpha\geq\omega$, $A_\beta\cap (\piba)^{-1}(Q_\alpha)$ and
$B_\beta\cap (\piba)^{-1} (Q_\alpha)$ are complementary regular closed
subsets of $(\piba )^{-1} (Q_\alpha)$.
\end{description}

The following claim and lemma explain our use of irreducible maps,
and the use of measure as a tool to achieve the HS+HL properties
of the space. The proof is basically the same as in \cite{DzK} but
we give it here since it explains the main point and also to show
how our situation actually simplifies the proof from  \cite{DzK}.
For any $\alpha$, we use the notation $[s]$ for a finite partial
function $s$ from $\alpha$ to 2 to denote the basic clopen set
$\{f\in 2^{\alpha}:\,s\subseteq f\}$, or its relativization to a
subspace of $2^\alpha$, as it is clear from the
context.~\footnote{The notation also does not specify $\alpha$ but
again following the tradition, we shall rely on $\alpha$ being
clear from the context.}

\begin{Claim}\label{induction} Assume the requirements {\bf R1}-{\bf R5} and {\bf R7}. Then
for each $\beta\in [\alpha,\omega_1]$ the projection $\piba$ is irreducible on
$(\piba )^{-1}  (Q_\alpha )$.
\end{Claim}

\begin{Proof of the Claim}
We use induction on $\beta\ge\alpha$.

The step $\beta=\alpha$ is clear. Assume that we know that the projection $\piba$ is irreducible on
$(\piba )^{-1}  (Q_\alpha )$ and let us prove that $\pi^{\beta+1}_\alpha$ is irreducible on
$(\pi^{\beta+1}_\alpha )^{-1}  (Q_\alpha )$.  Suppose that $F$ is a proper closed subset of $(\pi^{\beta+1}_\alpha )^{-1}
(Q_\alpha )$ satisfying $\pi^{\beta+1}_\alpha(F)=Q_\alpha$. Then by the inductive assumption
$\pi^{\beta+1}_\beta(F)=(\piba )^{-1}  (Q_\alpha )$.  Let $x\in (\pi^{\beta+1}_\alpha )^{-1}
(Q_\alpha )\setminus F$, so we must have that $x\rest\beta=s_\beta$. Assume $x(\beta)=0$, the case
$x(\beta)=1$ is symmetric. Because $F$ is closed, we can find a basic clopen set $[t]$ in $K_{\beta+1}$
containing $x$ such that $[t]\cap F=\emptyset$. Let $s=t\rest\beta$.

Therefore $s_\beta\in [s]$ holds in $K_\beta$, and by {\bf R7} we can find $y\in {\rm int}(A_\beta
\cap (\piba )^{-1}  (Q_\alpha ))\cap [s]$. Using the inductive assumption we conclude
$y\in {\rm int}(A_\beta\cap (\pi^{\beta+1}_\beta ) (F ))\cap [s]$, so there is a basic clopen
set $[v]\subseteq [s]$ in $K_\beta$ such that $y\in [v]$ and $[v]\subseteq A_\beta\cap (\pi^{\beta+1}_\beta ) (F )$.
But then $[v]$ viewed as a clopen set in $K_{\beta+1}$ satisfies $[v]\subseteq [t]$ and yet
$[v]\cap F\neq \emptyset$.

The limit case of the induction is easy by the definition of inverse limits.

$\eop_{\ref{induction}}$
\end{Proof of the Claim}

\begin{Lemma}\label{repeat} Assume the requirements {\bf R1}-{\bf R7}
and let $H$ be a closed subset of $K$.  Then
there is an $\alpha <\omega_1$ such that
$\pioa$ is irreducible on $(\pioa )^{-1} (\pioa (H))$.
\end{Lemma}

\begin{Proof of the Lemma}
For each $\gamma < \omega_1$, let $H_\gamma = \pi^{\omega_1}_\gamma (H)$.
Then the $\mu_\gamma (H_\gamma )$ form a non-increasing sequence
of real numbers, so we may fix a $\gamma <\omega_1$
such that for all $\alpha\geq\gamma$, $\mu_\alpha(H_\alpha)=
\mu_\gamma(H_\gamma)$.
Next fix an $\alpha \geq\gamma$ such that $\delta_\alpha = \gamma$ and
$J_\alpha = H_\gamma$.
Then $L_\alpha = C_\alpha = (\pi^\alpha_\gamma)^{-1} (H_\gamma)$.
Hence $H_\alpha$ is a closed
subset of $L_\alpha$ with the same measure as $L_\alpha$,
so $Q_\alpha \subseteq H_\alpha \subseteq L_\alpha$, by the definition of $Q_\alpha$. Recall that by Claim \ref{induction}
we have that $\pioa$ is irreducible on $(\pioa)^{-1}(Q_\alpha)$.

Now we claim that
$\pioa$ is 1-1 on $(\pioa )^{-1}(H_\alpha\setminus Q_\alpha)$. Otherwise, there would be $x\neq y\in
(\pioa )^{-1}(H_\alpha\setminus Q_\alpha)$ with $x\rest\alpha=y\rest\alpha$. Therefore for some
$\beta\ge\alpha$ we have $x\rest\beta=y\rest \beta=s_\beta$, as otherwise $x=(\pioa )^{-1}(\{x\rest\alpha\})$. In
particular  $s_\beta\in (\piba)^{-1}(H_\alpha)\subseteq  (\piba)^{-1}(L_\alpha)$. On the other hand, if
$s_\beta\in (\piba)^{-1}(Q_\alpha)$ then $\{x,y\}\in (\pioa )^{-1}(Q_\alpha)$- a contradiction- so
$s_\beta\notin (\piba)^{-1}(Q_\alpha)$.
This means
$s_\beta\in (\piba)^{-1}(N_\alpha)$, in contradiction with {\bf R6}.

Thus, $\pioa$ must be irreducible on $(\pioa )^{-1} (H_\alpha )$ as well, and the Lemma is proved.
$\eop_{\ref{repeat}}$
\end{Proof of the Lemma}

Now we comment on how to assure that $K$ is not Rosenthal compact.
A remarkable theorem of Todor\v cevi\'c from \cite{StevoRos}
states that every non-metric Rosenthal compactum contains either
an uncountable discrete subspace or a homeomorphic copy of the
split interval. As our $K$, being HS+HL, cannot have an
uncountable discrete subspace, it will suffice to show that it
does not contain a homeomorphic copy of the split interval.

\begin{Claim}\label{Rosenthal} Suppose that the requirements {\bf R1}-{\bf R7} are met. Then
\begin{description}
\item{(1)} all $\mu$-measure 0 sets in $K$ are second countable and
\item{(2)} $K$ does not contain a homeomorphic copy of the split interval.
\end{description}
\end{Claim}

\begin{Proof of the Claim}
(1) Suppose that $M$ is a $\mu$-measure 0 Borel set in $K$ and let
$N=\pi_\omega^{\omega_1}(M)$, hence $N$ is of measure 0 in
$2^\omega$. Let $\alpha\in[\omega,\omega_1)$ be such that
$\delta_\alpha= \omega$ and $J_\alpha=N$. Then
$C_\alpha=(\pi_\omega^\alpha)^{-1}(N)$ and hence
$\mu_\alpha(C_\alpha)=0$ and so $C_\alpha\subseteq N_\alpha$.
Requirement {\bf R6} implies that for $\beta\ge\alpha$,
$(\pi_\beta^{\omega_1})^{-1}(s_\beta)\cap M= \emptyset$, so the
topology on $M$ is generated by the basic clopen sets of the form
$[s]$ for $\dom(s)\subseteq\alpha$. So $M$ is 2nd countable.

(2) Suppose that $H\subseteq K$ is homeomorphic to the split interval. Therefore $H$ is compact
and therefore closed in $K$. In particular $\mu(H)$ is defined.

If $\mu(H)=0$ then by (1), $H$ is 2nd countable, a contradiction.
If $\mu(H)>0$ then there is an uncountable set $N\subseteq H$ with
$\mu(N)=0$. Then $N$ is uncountable and 2nd countable,
contradicting the fact that all 2nd countable subspaces of the
split interval are countable. $\eop_{\ref{Rosenthal}}$
\end{Proof of the Claim}

Now we comment on how we assure that any uncountable nice semi-biorthogonal
system in $C(K)$ is $\omega$-determined, i.e. any uncountable semi-bidiscrete sequence in $K$
forms an $\omega$-determined family of pairs of points. For this we make one further
requirement:

\begin{description}
\item{\bf R8}. If $\alpha\,,\beta\in[\omega,\omega_1)$ with $\alpha < \beta$ then $s_\beta\rest\alpha\neq s_\alpha$.
\end{description}


\begin{claim}
\label{nosupernice} Requirements {\bf R1}-{\bf R8} guarantee that
any uncountable semi-bidiscrete sequence in $K$ is
$\omega$-determined.\end{claim}

\begin{Proof of the Claim}
Suppose that $\langle
(x^0_\alpha,x^1_\alpha):\,\alpha<\omega_1\rangle$  forms an
uncountable semi-bidiscrete sequence in $K$ that is not
$\omega$-determined. By the definition of a semi-bidiscrete
sequence, the $ (x^0_\alpha,x^1_\alpha)$'s are distinct pairs of
distinct points. Therefore there must be $s\in 2^\omega$ such that
$A=\{\alpha:\,x^0_\alpha\rest\omega=x^1_\alpha\rest\omega=s\}$ is
uncountable. We have at least one $l < 2$ such that
$\{x^l_\alpha:\,\alpha\in A\}$ is uncountable, so assume, without
loss of generality, that this is true for $l=0$.

Let $\alpha, \beta,\gamma$ be three distinct members of $A$. Then
by Observation \ref{delta} we have $$x^0_\alpha\rest
\Delta(x^0_\alpha, x^0_\beta)=x^0_\beta\rest \Delta(x^0_\alpha,
x^0_\beta) =s_{ \Delta(x^0_\alpha, x^0_\beta)}$$ and similarly
$$x^0_\alpha\rest \Delta(x^0_\alpha, x^0_\gamma)=x^0_\gamma\rest
\Delta(x^0_\alpha, x^0_\gamma)=s_{ \Delta(x^0_\alpha,
x^0_\gamma)}.$$ By {\bf R8} we conclude that $\Delta(x^0_\alpha,
x^0_\beta)$ is the same for all $\beta \in A\setminus\{\alpha\}$
and we denote this common value by $\Delta_\alpha$. Thus for
$\beta \in A\setminus\{\alpha\}$ we have $x^0_\beta\rest
\Delta_\alpha=s_{\Delta_\alpha}$, but applying the same reasoning
to $\beta$ we obtain $x^0_\alpha\rest
\Delta_\beta=s_{\Delta_\beta}$ and hence by {\bf R8} again we have
$\Delta_\alpha=\Delta_\beta$. Let $\delta^\ast$ denote the common
value of $\Delta_\alpha$ for $\alpha \in A$.

Again, taking distinct $\alpha,\beta,\gamma\in A$ we have
$x^0_\alpha\rest\delta^\ast=
x^0_\beta\rest\delta^\ast=x^0_\gamma\rest\delta^\ast$ and that
$x^0_\alpha(\delta^\ast), x^0_\beta(\delta^\ast)$ and
$x^0_\gamma(\delta^\ast)$ are pairwise distinct. This is, however,
impossible as the latter have values in $\{0,1\}$.
$\eop_{\ref{nosupernice}}$
\end{Proof of the Claim}

Finally we show that the space $K$ is a 2-to-1 continuous preimage
of a compact metric space. We simply define
$\varphi:\,K\into2^\omega$ as $\varphi(x)=x\rest\omega$. This is
clearly continuous. To show that it is 2-to-1 we first prove the
following:

\begin{Claim}\label{ultra} In the space $K$ above, for any $\alpha\neq\beta$ we have
$s_\alpha\rest\omega\neq s_\beta\rest\omega$.
\end{Claim}

\begin{Proof of the Claim} Otherwise suppose that $\alpha<\beta$ and yet
$s_\alpha\rest\omega= s_\beta\rest\omega$. By {\bf R8} we have
$s_\alpha\nsubseteq s_\beta$, so
$\omega\le\delta=\Delta(s_\alpha,s_\beta)<\beta$. By Observation
\ref{delta} applied to any $x\supseteq s_\alpha$ and $y\supseteq
s_\beta$ from $K$, we have
$s_\alpha\rest\delta=x\rest\delta=y\rest\delta=s_\beta\rest\delta=s_\delta$.
But this would imply $s_\delta\subseteq s_\beta$, contradicting
{\bf R8}. $\eop_{\ref{ultra}}$
\end{Proof of the Claim}

Now suppose that $\varphi$ is not 2-to-1, that is there are three
elements $x,y,z \in K$ such that $x\rest\omega=y\rest\omega=
z\rest\omega$. Let $\alpha=\delta(x,y)$ and $\beta=\delta(x,z)$,
so $\alpha, \beta \ge\omega$. By Observation \ref{delta} we have
$x\rest\alpha=y\rest\alpha=s_\alpha$,
$x\rest\beta=z\rest\beta=s_\beta$, so by requirement {\bf R8} we
conclude $\alpha=\beta$. Note that then $y(\alpha)=z(\alpha)$ and
so $\delta=\Delta(y,z)>\alpha$ and $y\rest\delta=
s_\delta\supseteq s_\beta$, in contradiction with {\bf R8}.
Therefore $\varphi$ is really 2-to-1.

\subsubsection{Meeting the requirements}
Now we show how to meet all these requirements. It suffices to
show what to do at any successor stage $\alpha+1$ for $\alpha\in
[\omega,\omega_1)$, assuming all the requirements have been met at
previous stages.

First we choose $s_\alpha$. By {\bf R5} for any $\gamma<\alpha$ we
have $\mu_\gamma(\{s_\gamma\})=0$ and
$\mu_\alpha((\pi_\gamma^\alpha)^{-1}(s_\gamma))=0$. Hence the set
of points $s\in K_\alpha$ for which $s\rest \gamma=s_\gamma$ for
some $\gamma<\alpha$ has measure 0, so we simply choose $s_\alpha$
outside of
$\bigcup_{\gamma<\alpha}(\pi^\alpha_\gamma)^{-1}(s_\gamma)$ , as
well as outside of
$\bigcup_{\gamma<\alpha}(\pi^\alpha_\gamma)^{-1}(N_\gamma)$ (to
meet {\bf R6}), which is possible as the $\mu_\alpha$ measure of
the latter set is also 0.

Now we shall use an idea from \cite{DzK}. We fix a strictly
decreasing sequence $\langle V_n:\,n\in\omega\rangle$ of clopen
sets in $K_\alpha$ such that $V_0=K_\alpha$ and
$\bigcap_{n<\omega}V_n=\{s_\alpha\}$. We shall choose a function
$f:\,\omega\to\omega$ such that letting
$$A_\alpha=\bigcup_{n<\omega} (V_{f(2n)}\setminus
V_{f(2n+1)})\cup\{s_\alpha\}$$ and $$B_\alpha=\bigcup_{n<\omega}
(V_{f(2n+1)}\setminus V_{f(2n)})\cup\{s_\alpha\}$$ will meet all
the requirements. Once we have chosen $A_\alpha$ and $B_\alpha$,
we let $$K_{\alpha+1}=A_\alpha\times \{0\}\cup B_\alpha\times
\{1\}.$$ For a basic clopen set $[s]=\{g\in
K_{\alpha+1}:\,g\supseteq s\}$, where $s$ is a finite partial
function from $\alpha+1$ to 2 and $\alpha\in\dom(s)$, we let
$\mu_{\alpha+1}([s])=1/2\cdot \mu_\alpha([s\rest\alpha])$. We
prove below that this extends uniquely to a Baire measure on
$K_{\alpha+1}$.

The following is basically the same (in fact simpler) argument
which appears in \cite{DzK}. We state and prove it here for the
convenience of the reader.

\begin{Claim}\label{fastfunction} The above choices of $A_\alpha$, $B_\alpha$, and
$\mu_{\alpha+1}$, with the choice of any function $f$ which is
increasing fast enough, will satisfy all the requirements {\bf
R1}-{\bf R8}.
\end{Claim}

\begin{Proof of the Claim} Requirements {\bf R1}-{\bf R4} are clearly met with any choice of $f$.
To see that {\bf R5} is met, let us first prove that
$\mu_{\alpha+1}$ as defined above indeed extends uniquely to a
Baire measure on $K_{\alpha+1}$. We have already defined
$\mu_{\alpha+1}([s])$ for $s$ satisfying $\alpha\in\dom(s)$. If
$\alpha\notin\dom(s)$ then we let
$\mu_{\alpha+1}([s])=\mu_{\alpha}(\pi^{\alpha+1}_\alpha [s])$. It
is easily seen that this is a finitely additive measure on the
basic clopen sets, which then extends uniquely to a Baire measure
on $K_{\alpha+1}$. It is also clear that this extension satisfies
{\bf R5}.

Requirements {\bf R6} and {\bf R8} are met by the choice of
$s_\alpha$, so it remains to see that we can meet {\bf R7}. For
each $\gamma\in[\omega,\alpha]$, if $s_\alpha\in
(\pi_\gamma^\alpha)^{-1}(Q_\gamma)$, fix an $\omega$-sequence
$\bar{t}_\gamma$ of distinct points in
$(\pi_\gamma^\alpha)^{-1}(Q_\gamma)$ converging to $s_\alpha$.
Suppose that $\bar{t}_\gamma$ is defined and that both
$A_\alpha\setminus B_\alpha$ and $B_\alpha\setminus A_\alpha$
contain infinitely many points from $\bar{t}_\gamma$. Then we
claim that $A_\alpha\cap (\pi_\gamma^\alpha)^{-1}(Q_\gamma)$ and
$B_\alpha\cap (\pi_\gamma^\alpha)^{-1} (Q_\gamma)$ are
complementary regular closed subsets of $(\pi_\gamma^\alpha )^{-1}
(Q_\gamma)$. Note that we have already observed that $Q_\gamma$
does not have isolated points, so neither does
$(\pi_\gamma^\alpha)^{-1} (Q_\gamma)$. Hence, since
$\{s_\alpha\}\supseteq A_\alpha\cap
(\pi_\gamma^\alpha)^{-1}(Q_\gamma)\cap B_\alpha\cap
(\pi_\gamma^\alpha)^{-1} (Q_\gamma)$, we may conclude that this
intersection is nowhere dense in both $A_\alpha\cap
(\pi_\gamma^\alpha)^{-1}(Q_\gamma)$ and $B_\alpha\cap
(\pi_\gamma^\alpha)^{-1} (Q_\gamma)$. Finally, $A_\alpha\cap
(\pi_\gamma^\alpha)^{-1}(Q_\gamma)$ and $B_\alpha\cap
(\pi_\gamma^\alpha)^{-1} (Q_\gamma)$ are regular closed because we
have assured that $s_\alpha$ is in the closure of both.

Therefore we need to choose $f$ so that for every relevant
$\gamma$, both $A_\alpha\setminus B_\alpha$ and $B_\alpha\setminus
A_\alpha$ contain infinitely many points of $\bar{t}_\gamma$.
Enumerate all the relevant sequences $\bar{t}_\gamma$ as
$\{\bar{z}^k\}_{k<\omega}$. Our aim will be achieved by choosing
$f$ in such a way that, for every $n$, both sets
$V_{f(2n)}\setminus V_{f(2n+1)}$ and $V_{f(2n+1)}\setminus
V_{f(2n+2)}$ contain a point of each $\bar{z}^k$ for $k\le n$.
$\eop_{\ref{fastfunction}}$
\end{Proof of the Claim}

This finishes the proof of the theorem.
$\eop_{\ref{CH}}$
\end{proof}

\section{Bidiscrete systems}\label{bidiscr} The main result of this section is Theorem \ref{lemma2}
below.  In the course of proving Theorem 10 in \S7 of
\cite{StevoMM}, Todor\v cevi\'c actually proved that if $K$ is not
hereditarily separable then it has an uncountable bidiscrete
system. Thus his proof yields Theorem \ref{lemma2} for
$d(K)=\aleph_1$ and the same argument can be easily extended to a
full proof of \ref{lemma2}.

Let us first state some general observations about bidiscrete systems.

\begin{observation}
\label{closedsubspace}
Suppose that $K$ is a compact Hausdorff space and $H\subseteq K$ is closed, while
$\{ (x_\alpha^0, x_\alpha^1):\,\alpha<\kappa\}$
is a bidiscrete system in $H$, as exemplified by functions $f_\alpha\,(\alpha<\kappa)$.
Then there are functions $g_\alpha\,(\alpha<\kappa)$ in $C(K)$
such that $f_\alpha\subseteq g_\alpha$ and $g_\alpha\,(\alpha<\kappa)$ exemplify that
 $\{ (x_\alpha^0, x_\alpha^1):\,\alpha<\kappa\}$
is a bidiscrete system in $K$.
\end{observation}

\begin{Proof} Since $H$ is closed we can, by Tietze's Extension Theorem,
extend each $f_\alpha$ continuously to
a function $g_\alpha$ on $K$. The conclusion follows
from the definition of a bidiscrete system.
$\eop_{\ref{closedsubspace}}$
\end{Proof}

\begin{Claim}\label{generaldiscrete} Suppose that $K$ is a compact
space and $F_i\subseteq G_i \subseteq K\,$ for $i\in I$ are such
that the $G_i$'s are disjoint open, the $F_i$'s are closed and in
each $F_i$ we have a bidiscrete system $S_i$. Then $\bigcup_{i\in
I}S_i$ is a bidiscrete system in $K$.
\end{Claim}

\begin{proof} For $i\in I$ let the bidiscreteness of $S_i$ be witnessed
by $\{g^i_{\alpha}\,:\,\alpha<\kappa_i\} \subseteq C(F_i)$. We can,
as in Observation \ref{closedsubspace}, extend each $g^i_{\alpha}$
to $h^i_{\alpha} \in C(K)$ which exemplify that $S_i$ is a
bidiscrete system in $K$.

Now we would like to put all these bidscrete systems together, for
which we need to find appropriate witnessing functions. For any
$i\in I$ we can apply Urysohn's Lemma to find functions $f_i \in
C(K)$ such that $f_i$ is 1 on $F_i$ and 0 on the complement of
$G_i$. Let us then put, for any $\alpha$ and $i$,
$f_{\alpha}^i=g^i_{\alpha}\cdot f_i$. Now, it is easy to verify
that the functions $\{f^i_{\alpha}:\,\alpha<\kappa_i, i\in I\}$
witness that $\bigcup_{i\in I}S_i$ is a bidiscrete system in $K$.
$\eop_{\ref{generaldiscrete}}$
\end{proof}

Clearly, Observation \ref{closedsubspace} is the special case of
Claim \ref{generaldiscrete} when $I$ is a singleton and $G_i = K$.

\begin{Claim}\label{splitintervalsplit}
If the compact space $K$ has a discrete subspace of size $\kappa
\ge \omega$ then it has a bidiscrete system of size $\kappa$, as
well.
\end{Claim}

\begin{Proof}
Suppose that $D = \{x_\alpha : \alpha < \kappa\}$ (enumerated in a
one-to-one manner) is discrete in $K$ with open sets $U_\alpha$
witnessing this, i.e. $D \cap U_\alpha = \{x_\alpha\}$ for all
$\alpha < \kappa$. For any $\alpha < \kappa$ we may fix a function
$f_\alpha \in C(K)$ such that $f_\alpha(x_{2\alpha+1}) = 1$ and
$f_\alpha(x) = 0$ for all $x \notin U_{2\alpha+1}$. Obviously,
then $\{f_\alpha : \alpha < \kappa\}$ exemplifies that
$\{(x_{2\alpha},x_{2\alpha+1}) : \alpha < \kappa \}$ is a
bidiscrete system in $K$.
\end{Proof}

The converse of Claim \ref{splitintervalsplit} is false, however
the following is true.
\begin{claim}
\label{discbidisc} Suppose that $B = \{ (x^0_\alpha,
x^1_\alpha):\,\alpha<\kappa\}$ is a bidiscrete system in $K$. Then
$B$ is a discrete subspace of $K^2$.
\end{claim}

\begin{Proof} Assume that the functions $\{f_\alpha:\,\alpha<\kappa\}
\subseteq C(K)$ exemplify the  bidiscreteness of $B$. Then
$O_\alpha=f_\alpha^{-1}((-\infty,1/2))\times
f_\alpha^{-1}((1/2,\infty))$ is an open set in $K^2$ containing
$(x^0_\alpha, x^1_\alpha)$. Also, if $\beta\neq\alpha$ then
$(x^0_\beta, x^1_\beta) \notin O_\alpha$, hence $B$ is a discrete
subspace of $K^2$. $\eop_{\ref{discbidisc}}$
\end{Proof}

Now we turn to formulating and proving the main result of this
section.

\begin{Theorem}\label{lemma2} If $K$ is an infinite compact Hausdorff space
then $K$ has a bidiscrete system of size $d(K)$. If $K$ is moreover 0-dimensional
then there is a very nice bidiscrete system in $K$ of size $d(K)$.
\end{Theorem}

\begin{Proof} The proofs of the two parts of the theorem are the same, except that
in the case of a 0-dimensional space every time
that we take functions witnessing bidiscreteness, we need to observe that these functions
can be assumed to take values only in $\{0,1\}$. We leave it to the reader to check that this is indeed
the case.

The case $d(K)=\aleph_0$ is very easy, as it is well known that
every infinite Hausdorff space has an infinite discrete subspace
and so we can apply Claim \ref{splitintervalsplit}. So, from now
on we assume that $d(K)>\aleph_0$.

Recall that a Hausdorff space $(Y,\sigma)$ is said to be {\em
minimal Hausdorff} provided that there does not exist another
Hausdorff topology $\rho$ on $Y$ such that $\rho\subsetneq
\sigma$, i.e. $\rho$ is strictly coarser than $\sigma$. The
following fact is well known and easy to prove, and it will
provide a key part of our argument:

\begin{Fact}\label{coarse} Any compact Hausdorff space is minimal Hausdorff.
\end{Fact}

\begin{Lemma}\label{lemma1} Suppose that $X$ is a compact Hausdorff space with $d(X)\ge\kappa>\aleph_0$
in which every non-empty open (equivalently: regular closed)
subspace has weight $\ge\kappa$.
Then $X$ has a bidiscrete system of size $\kappa$.
\end{Lemma}

\begin{Proof of the Lemma} We shall choose $x_\alpha^0, x_\alpha^1, f_\alpha$ by induction on $\alpha<\kappa$
so that the pairs $(x_\alpha^0, x_\alpha^1)$ form a bidiscrete
system, as exemplified by the functions $f_\alpha$. Suppose that
$x_\beta^0, x_\beta^1, f_\beta$ have been chosen for
$\beta<\alpha<\kappa$.

Let $C_\alpha$ be the closure of the set $\{x_\beta^0,
x_\beta^1:\,\beta<\alpha\}$. Therefore $d(C_\alpha)<\kappa$ and,
in particular, $C_\alpha\neq X$. Let $F_\alpha\subseteq X\setminus
C_\alpha$ be non-empty regular closed, hence
$w(F_\alpha)\ge\kappa$.

Let $\tau_\alpha$ be the topology on $F_\alpha$ generated by the
family
\[
\mathcal{F}_\alpha = \{f_\beta^{-1}(-\infty,q)\cap
F_\alpha\,,\,f_\beta^{-1}(q,\infty)\cap F_\alpha\,
:\,\beta<\alpha,\, q\in {\mathbf Q}\},
 \]
where $\mathbf Q$ denotes the set of rational numbers. Then
$|\mathcal{F}_\alpha|<\kappa$ (as $\kappa>\aleph_0$), hence the
weight of $\tau_\alpha$ is less than $\kappa$, consequently
$\tau_\alpha$ is strictly coarser than the subspace topology on
$F_\alpha$. Fact \ref{coarse} implies that $\tau_\alpha$ is not a
Hausdorff topology on $F_\alpha$, hence we can find two distinct
points $x^0_\alpha, x^1_\alpha\in F_\alpha$ which are not
$T_2$-separated by any two disjoint sets in $\tau_\alpha$ and, in
particular, in $\mathcal{F}_\alpha$. This clearly implies that
$f_\beta(x^0_\alpha)=f_\beta(x^1_\alpha) $ for all $\beta<\alpha$.

Now we use the complete regularity of $X$ to find $f_\alpha\in
C(X)$ such that $f_\alpha$ is identically 0 on the closed set
$C_\alpha\cup\{x^0_\alpha\}$ and $f_\alpha(x^1_\alpha)=1$. It is
straight-forward to check that $\{f_\alpha : \alpha < \kappa\}$
indeed witnesses the bidiscreteness of $\{(x_\alpha^0, x_\alpha^1)
: \alpha < \kappa\}$. $\eop_{\ref{lemma1}}$
\end{Proof of the Lemma}

Let us now continue the proof of the theorem. We let $\kappa$
stand for $d(K)$ and let
\[
{\mathcal P}=\{\emptyset\neq O\subseteq K:\,O\mbox{ open such that }[\emptyset\neq U
\mbox{ open}\subseteq O\implies d(U)= d(O)]\}.
\]
We claim that ${\mathcal P}$ is a $\pi$-base for $K$, i.e. that
every non-empty open set includes an element of ${\mathcal P}$.
Indeed, suppose this is not case, as witnessed by a non-empty open
set $U_0$. Then $U_0\notin\PP$, so there is a non-empty open set
$\emptyset\neq U_1\subseteq U_0$ with $d(U_1)<d(U_0)$ (the case
$d(U_1)<d(U_0)$ cannot occur). Then $U_1$ itself is not a member
of $\PP$ and therefore we can find a non-empty open set
$\emptyset\neq U_2\subseteq U_1$ with $d(U_2)<d(U_1)$, etc. In
this way we would obtain an infinite decreasing sequence of
cardinals, a contradiction.

Let now $\mathcal O$ be a maximal disjoint family of members of
${\mathcal P}$. Since ${\mathcal P}$ is a $\pi$-base for $K$ the
union of $\mathcal{O}$ is clearly dense in $K$. This implies that
if we fix any dense subset $D_O$ of $O$ for all $O\in {\mathcal
O}$ then $\bigcup \{ D_O : O \in \mathcal{ O} \}$ is dense in $K$,
as well. This, in turn, implies that $\sum \{d(O) : O \in \mathcal
{O}\} \ge d(K) = \kappa$.

If $|\mathcal{O}| = \kappa$ then we can select a discrete subspace
of $K$ of size $\kappa$ by choosing a point in each $O\in \mathcal
O$, so the conclusion of our theorem follows by Corollary
\ref{splitintervalsplit}.

So now we may assume that $|\mathcal O|<\kappa$. In this case,
since $\kappa>\aleph_0$,  letting $\OO'=\{O\in
\OO:\,d(O)>\aleph_0\}$, we still have $\sum \{d(O) : O \in
\mathcal {O'}\} \ge \kappa$. Next, for each $O\in\OO'$ we choose a
non-empty open set $G_O$ such that its closure
$\overline{G}_O\subseteq O$. Then we have, by the definition of
$\PP$, that $d(\overline{G}_O)=d(G_O)=d(O)$. By the same token,
every non-empty open subspace of the compact space
$\overline{G}_O$ has density $d(O)$, and hence weight $\ge d(O)$.
Therefore we may apply Lemma \ref{lemma1} to produce a bidiscrete
system $S_O$ of size $d(O)$ in $\overline{G}_O$. But then Claim
\ref{generaldiscrete} enables us to put these systems together to
obtain the bidscrete system $S = \bigcup \{S_O : O \in
\mathcal{O}'\}$ in $K$ of size $\sum \{d(O) : O \in \mathcal
{O'}\} \ge\kappa$. $\eop_{\ref{lemma2}}$
\end{Proof}

It is immediate from Theorem \ref{lemma2} and Observation
\ref{closedsubspace} that if $C$ is a closed subspace of the
compactum $K$ with $d(C) = \kappa$ then $K$ has a bidiscrete
system of size $\kappa$. We recall that the hereditary density
${\rm hd}(X)$ of a space $X$ is defined as the supremum of the
densities of all subspaces of $X$.

\begin{Fact}\label{hd} For any compact Hausdorff space $K$,
${\rm hd}(K)=\sup\{d(C):\,C\mbox{ closed}\subseteq K\}$.
\end{Fact}

From this fact and what we said above we immediately obtain the
following corollary of Theorem \ref{lemma2}.

\begin{Corollary}\label{theorem2} If $K$ is a compact Hausdorff space with
${\rm hd}(K)\ge\lambda^+$ for some $\lambda\ge\omega$, then $K$
has a bidiscrete system of size $\lambda^+$.
\end{Corollary}

\bigskip

We finish by listing some open questions.

\begin{Question} (1) Does every compact space $K$ admit a
bidiscrete system of size ${\rm hd}(K)$?

{\noindent (2)} Define
\[
{\rm bd}(K)=\sup\{|S|\,:\, S \mbox{ is a bidiscrete system in
}K\}.
\]
Is there always a bidiscrete system in $K$ of size ${\rm bd}(K)$?

{\noindent (3)} Suppose that $K$ is a 0-dimensional compact space
which has a bidiscrete system of size $\kappa$. Does then $K$ also
have a very nice bidiscrete system of size $\kappa$ (i.e. such
that the witnessing functions take values only in $\{0,1\}$)? Is
it true that any bidiscrete system in a 0-dimensional compact
space is very nice?

{\noindent (4)} (This is Problem 4 from \cite{JuSz}): Is there a
$ZFC$ example of a compact space $K$ that has no discrete subspace
of size $d(K)$?

{\noindent (5)} If the square $K^2$ of a compact space $K$
contains a discrete subspace of size $\kappa$, does then $K$ admit
a bidiscrete system of size $\kappa$ (or does at least $C(K)$ have
a biorthogonal system of size $\kappa$)? This question is of
especial interest for $\kappa = \omega_1$.
\end{Question}


\end{document}